\documentclass[12pt, twoside, leqno]{article}

\usepackage{amsmath,amsthm}
\usepackage{amssymb}

\usepackage{enumitem}

\usepackage{graphicx}

\usepackage[T1]{fontenc}

\usepackage{multirow}

\newtheorem{theorem}{Theorem}[section]

\newtheorem{lemma}[theorem]{Lemma}
\newtheorem{proposition}[theorem]{Proposition}
\newtheorem{problem}[theorem]{Problem}

\newtheorem{mainthm}[theorem]{Main Theorem}

\theoremstyle{definition}
\newtheorem{definition}[theorem]{Definition}

\newtheorem{example}[theorem]{Example}

\newtheorem*{xrem}{Remark}

\numberwithin{equation}{section}

\DeclareMathOperator{\len}{length}

\newcommand{\obj}[3]{\mathcal{F}^{#1}\mathbb{S}^{#2}\mathbf{G}_{#3}}

\def\F{{\mathbb F}}

\def\PSL2{PSL_2 (\mathbb{Z})}
\def\SL2{SL_2 (\mathbb{Z})}

\def\genset{\mathcal{X}}

\def\monoid{\left<\genset\right>}
\def\lbrack{\left[}
\def\rbrack{\right]}

\def\ad{{\rm ad}\ }

\def\indset{\mathcal{I}}

\def\indsetL4{\indset^*}

\def\Basis5{\mathcal{B}}
\def\Reps5{\mathcal{S}}

\def\freeassoc3{\F\left< A,B,C\right>}

\def\LieAB{\lbrack A,B\rbrack}
\def\algI{I}

\def\theAlg{\F\monoid/\idcomm_\sigma}

\def\idcomm{\mathcal{I}}

\def\newgenset{\mathcal{Y}}
\def\newidcomm{\mathcal{J}}
\def\newmonoid{\left<\newgenset\right>}

\def\redsys{\Gamma}

\def\id{\texttt{id}}

\def\Csub{\F[C]}

\def\overlap{\Phi}

\def\newAlg{\F\newmonoid/\newidcomm_\sigma}



\begin{document}


\baselineskip=17pt


\title{Lie polynomials in an algebra defined by a linearly twisted commutation relation}

\author{Rafael Reno S. Cantuba\\
Mathematics and Statistics Department\\ 
De La Salle University, Manila\\
2401 Taft Ave.\\
1004 Manila, Philippines\\
E-mail: rafael\_cantuba@dlsu.edu.ph}

\date{}

\maketitle


\renewcommand{\thefootnote}{}



\renewcommand{\thefootnote}{\arabic{footnote}}
\setcounter{footnote}{0}


\begin{abstract}
We present an elementary approach in characterizing Lie polynomials in the generators $A,B$ of an algebra with a defining relation that is in the form of a deformed or twisted commutation relation $AB=\sigma(BA)$ where the deformation or twisting map $\sigma$ is a linear polynomial with a slope parameter that is not a root of unity. The class of algebras defined as such encompasses $q$-deformed Heisenberg algebras, rotation algebras, and some types of $q$-oscillator algebras whose deformation parameters are not roots of unity, and so we have a general solution for the Lie polynomial characterization problem for these algebras.
\end{abstract}

\section{Introduction}
\def\introAlg{\mathfrak{A}}
This work is inspired by the classical commutation relations in quantum physics such as the canonical commutation relation for bosonic systems \cite{Gaa54b}, the canonical anticommutation relation for fermionic systems \cite{Gaa54a}, the deformed commutation relation for $q$-oscillators \cite{Bie89,Hel00,Hel05,Mac89}, and also the related commutation relations for rotation algebras \cite{Eff67,Rie81,Web13}. The study of these commutation relations has led to the development of elegant results under Algebra and Functional Analysis. See for instance, \cite[pp. 217--220]{Bra97}. An abstract setting for all commutation relations mentioned is to view them as specific manifestations of a defining relation of a unital associative algebra $\introAlg$ generated by two elements $A$, $B$. We can express this defining relation for $\introAlg$ in the form $AB=\sigma(BA)$ for some function $\sigma$ defined on the free unital associative algebra on $A,B$ into itself, and this function $\sigma$, the deformation or twisting map, can be considered as ``linear'' in the sense of a polynomial of degree one. We call the relation $AB=\sigma(BA)$ as a \emph{deformed} or  \emph{twisted commutation relation}.

A natural Lie algebra structure is induced in $\introAlg$ by the operation $\lbrack P,Q\rbrack :=PQ-QP$ defined for any $P,Q\in\introAlg$. We are interested in the Lie subalgebra of $\introAlg$ generated by the same generators $A,B$, that is, the set of all Lie polynomials in $A,B$. Our goal is to obtain a characterization for any nonzero $P\in\introAlg$ to be a Lie polynomial in $A,B$, or alternatively, to determine a linear complement of the Lie subalgebra of $\introAlg$ generated by $A,B$. We shall call the problem of finding such a characterization (or the corresponding direct sum decomposition) as the \emph{Lie polynomial characterization problem} for the algebra $\introAlg$ (with respect to the presentation in which $A,B$ are generators). As we shall see, nontrivial forms of $AB=\sigma(BA)$ are inexpressible in terms of Lie algebra operations only, which means that the results from the theory of bases of finitely-generated and finitely-presented Lie algebras is not directly applicable, and so other  methods have to be devised in order to solve the Lie polynomial characterization problem.

The solution to the Lie polynomial characterization problem for a free algebra has resulted into an elegant theory that involves Hopf algebras. See for instance \cite[Theorem 1.4]{Reu93} in which the maps used in the Lie polynomial characterizations are precursors to establishing a Hopf algebra structure in the same algebra \cite[Proposition 1.10]{Reu93}. For algebras that are not free, we are interested in algebras with defining relations that are not (all) Lie polynomials. The first appearance of such a Lie polynomial characterization problem is \cite[Problem 12.14]{Ter11}, which is about the universal Askey-Wilson algebra $\Delta$. Initial progress was made in \cite{Can15}. However, the Lie polynomial characterization problem for $\Delta$ remains open. For $q$-deformed Heisenberg algebras with $q$ not a root of unity, the Lie polynomial characterization problem is solved \cite{Can17}, and this current paper serves as an initial step in generalizing \cite{Can17}. Another recent result is a solution to the Lie polynomial characterization problem for a toral algebra (an algebra related to a torus) in which the Fairlie-Odesskii algebra is embedded \cite{Can18}. The Lie polynomial characterization in the said toral algebra was used to prove that the Casimir element of the Fairlie-Odesskii algebra is not a Lie polynomial if the deformation parameter is not a root of unity \cite[Theorem 5.3]{Can18}.

The basis theorem for the algebra in this paper, Lemma~\ref{basisLem}, is a generalization of \cite[Corollary 4.5]{Hel05}, which is for the specific case of a $q$-deformed Heisenberg algebra. The proof of \cite[Corollary 4.5]{Hel05} makes use of gradations, degree functions, chain conditions on ideals, and explicit reordering formulae that involve $q$-special combinatorics \cite[Appendix C]{Hel00}. Our proof for the generalization, Lemma~\ref{basisLem}, is elementary and makes use only of \emph{the Diamond Lemma for Ring Theory} \cite[Theorem 1.2]{Ber78}. Our main result, Theorem~\ref{mainThm}, is a generalization of \cite[Theorem 5.8]{Can17}, but our proof of the former is also elementary compared to the proof of the latter from \cite{Can17}, which relies on the theory of the \emph{Lyndon-Shirshov basis} of a free Lie algebra, which can be found in sources such as the seminal works \cite{Shi53,Shi58}, and the excellent exposition \cite[Section 2.8]{Ufn95}.

\section{Preliminaries}
Let $\F$ denote a fixed but arbitrary field. Throughout, an algebra is assumed to be unital and associative. Given an algebra $\mathcal{A}$ over $\F$ with unity element $\algI_\mathcal{A}$, we use the convention that $U^0=\algI_\mathcal{A}$ for any $U\in\mathcal{A}$, and we denote $\algI_\mathcal{A}$ simply by $\algI$ if no confusion shall arise. Any subalgebra is assumed to contain the unity element. Thus, if a subalgebra $\mathcal{S}$ of a nontrivial algebra $\mathcal{A}$ is generated by only one element $X\in\mathcal{A}\backslash\{0,\algI\}$, we simply write $\mathcal{S}$ as the polynomial algebra $\F[X]\subseteq\mathcal{A}$. We also note that every algebra $\mathcal{A}$ has a Lie algebra structure induced by $\lbrack U,V\rbrack:=UV-VU$ for all $U,V\in\mathcal{A}$. Throughout, whenever we refer to the Lie algebra structure on an algebra $\mathcal{A}$, we shall always mean that which is induced by the Lie bracket operation $\lbrack\cdot,\cdot\rbrack$ just mentioned. Let $X_1,X_2,\ldots,X_n\in\mathcal{A}$. If $\mathcal{K}$ is the Lie subalgebra of $\mathcal{A}$ generated by $X_1,X_2,\ldots,X_n$, then we call the elements of $\mathcal{K}$ the \emph{Lie polynomials in $X_1,X_2,\ldots,X_n$}. Given $U\in\mathcal{A}$, the linear map $\ad U:\mathcal{A}\rightarrow\mathcal{A}$ is defined by the rule $V\mapsto\lbrack U,V\rbrack$. Also, for any linear map $\varphi$ whose domain and codomain are equal, and for any nonnegative integer $n$, by $\varphi^n$ we mean composition of $\varphi$ with itself $n$ times, where $\varphi^0$ is interpreted as the identity linear map $\id$. We interpret the composition of an empty collection of linear maps also as $\id$.

Given a nonempty finite set $\mathcal{S}$, we denote by $\left<\mathcal{S}\right>$ the free monoid on $\mathcal{S}$. We call the elements of $\left<\mathcal{S}\right>$ the \emph{words} on $\mathcal{S}$. The multiplicative operation on $\left<\mathcal{S}\right>$ is the concatenation product, and we call the multiplicative identity in $\left<\mathcal{S}\right>$ the \emph{empty word}. By $\F\left<\mathcal{S}\right>$ we mean the free algebra generated by $\mathcal{S}$. The set $\left<\mathcal{S}\right>$ is a basis for the vector space $\F\left<\mathcal{S}\right>$, and the empty word in $\left<\mathcal{S}\right>$ is the unity element of the algebra $\F\left<\mathcal{S}\right>$. 

Consider a two-element generating set $\genset:=\{A,B\}$. Given a function  $\sigma:\F\monoid\rightarrow\F\monoid$, denote by $\idcomm_\sigma$ the (two-sided) ideal of $\F\monoid$ generated by $AB-\sigma(BA)$. We call the defining relation $AB=\sigma(BA)$ of the quotient algebra $\F\monoid/\idcomm_\sigma$ as \emph{the $\sigma$-twisted commutation relation (satisfied by $A,B$)}. If there exist $m,b\in\F$ with $m\neq 0$ such that $\sigma(X)=mX+b\algI$, we say that the commutation relation $AB=\sigma(BA)$ is \emph{linearly twisted (with parameters $m,b$)}. We may also refer to $m$ as the \emph{slope parameter} of the twisting map $\sigma$. We further exclude the case $m=1$, where $\F\monoid/\idcomm_\sigma$ becomes the \emph{Heisenberg-Weyl algebra} \cite[Definition 2.1]{Goo15} over $\F$. Algebraic aspects such as reordering and normal forms in the Heisenberg-Weyl algebra are excellently discussed in \cite[Section 1.2]{Hel02} and also in \cite{Bla03}. However, the Lie polynomial characterization problem in the Heisenberg-Weyl algebra has a trivial solution. The set of all Lie polynomials in $A,B$ is the three-dimensional Lie algebra with basis $A,B,\LieAB$, and the corresponding commutator table for this basis is completely determined by the relations $\lbrack A,\LieAB\rbrack=0=\lbrack\LieAB, B\rbrack$.

From this point onward, assume that $AB=\sigma(BA)$ is linearly twisted with parameters $m,b$. By a simple result from the theory of free Lie algebras \cite[Lemma 1.7]{Reu93}, the generator $AB-\sigma(BA)$ of $\idcomm_\sigma$ is not an element of the free Lie algebra on $\genset$. That is, the defining relation of $\F\monoid/\idcomm_\sigma$ cannot be expressed in terms of Lie algebra operations only. Thus, the existing elegant and powerful machinery for computing a basis for finitely-presented Lie algebras, such as those in \cite{Ger96,Ger97a,Ger97b}, are not directly applicable, but in this work we present an elementary approach to the Lie polynomial characterization problem in  $\F\monoid/\idcomm_\sigma$.

\section{A reduction system for $\F\monoid/\idcomm_\sigma$ in three generators}

 As will become apparent in the succeeding parts of this section, we have good reason to consider properties of $\F\monoid/\idcomm_\sigma$ that are related to an algebra with three generators. Let $\newgenset=\{A,B,C\}$, and denote by $\newidcomm_\sigma$ the  ideal of $\F\newmonoid$ generated by $AB-mBA-b\algI$ and $C-AB+BA$. An immediate consequence is that there exists an algebra isomorphism $\F\monoid/\idcomm_\sigma\rightarrow\F\newmonoid/\newidcomm_\sigma$ such that $A\mapsto A$, $B\mapsto B$, and $\LieAB\mapsto C$. 
 
 Our goal in this section is to determine a basis for $\newAlg$ using the Diamond Lemma for Ring Theory \cite[Theorem 1.2]{Ber78}. We establish a sufficient condition for some elements of $\newAlg$ to form a basis for $\newAlg$, by invoking the Diamond Lemma. To establish this sufficient condition, we recall terminology and notation from \cite[Section 1]{Ber78}. We introduce each of such notions as the need for it arises. As a first step, we give an alternative presentation of $\newAlg$ that will turn out to be more suited to the Diamond Lemma.

\begin{proposition}\label{presProp} The algebra $\F\newmonoid/\newidcomm_\sigma$ has a presentation with generators $A,B,C$ and relations
\begin{eqnarray}
AB & = &\frac{mC-b\algI}{m-1},\label{redrel1}\\
AC & = & mCA,\label{redrel2}\\
BA & = &\frac{C-b\algI}{m-1},\label{redrel3} \\
CB & = &  mBC,\label{redrel4}\\
BC^kA & = & \frac{C^{k+1}-bC^k}{m^k(m-1)},\quad\quad k\in\{1,2,\ldots\}.\label{redrel5}
\end{eqnarray}
\end{proposition}
\def\oldgen{\zeta}
\def\newgen{\xi}
\begin{proof} View the left-hand side $AB$ and the right-hand side $\frac{mC-b\algI}{m-1}$ of \eqref{redrel1} as elements of the free algebra $\F\newmonoid$, and define $\newgen_1\in\F\newmonoid$ as the difference $AB-\frac{mC-b\algI}{m-1}$. In a similar manner, define $\newgen_2,\  \newgen_3,\  \newgen_4\in\F\newmonoid$ for the relations \eqref{redrel2} to \eqref{redrel4}, respectively. As for the countably infinite relations \eqref{redrel5}, for each positive integer $k$, we define $\xi_5(k)\in\F\newmonoid$ to be the left-hand side of \eqref{redrel5} minus the right-hand side. Also, denote the generators of $\newidcomm_\sigma$ by $\oldgen_1:=AB-mBA-b\algI$ and $\oldgen_2:=C-AB+BA$. Let $\mathcal{K}$ be the ideal of $\F\newmonoid$ generated by 
\begin{eqnarray}
\{\newgen_1,\newgen_2,\newgen_3,\newgen_4\}\cup\{\newgen_5(k)\  :\  k\in\{1,2,\ldots\}\}.
\end{eqnarray}
We claim that $\newidcomm_\sigma=\mathcal{K}$. By routine computations, it can be shown that the relations
\begin{eqnarray}
\oldgen_1 & = & \newgen_1-m\newgen_2,\label{eqId1}\\
\oldgen_2 & = &-\newgen_1+\newgen_2,\label{eqId2}\\
\newgen_1 & = &\frac{\oldgen_1+m\oldgen_2}{1-m},\label{eqId3}\\
\newgen_2 & = &\lbrack A,\oldgen_1\rbrack+A\oldgen_1-m\oldgen_2A,\label{eqId4}\\
\newgen_3 & = &\frac{\oldgen_1+\oldgen_2}{1-m},\label{eqId5}\\
\newgen_4 & = &\lbrack \oldgen_1,B\rbrack+\oldgen_2B-mB\oldgen_2,\label{eqId6}\\
\newgen_5(k) & = & m^{-k}\newgen_3-\sum_{i=1}^km^{i-1-k}C^{k-i}\newgen_4C^{i-1}A,\   (k\in\{1,2,\ldots\}),\label{eqId7}
\end{eqnarray}
hold in $\F\newmonoid$. We see from \eqref{eqId1} and \eqref{eqId2} that the generators $\oldgen_1$, $\oldgen_2$ of $\newidcomm_\sigma$ are elements of $\mathcal{K}$, and so $\newidcomm_\sigma\subseteq\mathcal{K}$. By \eqref{eqId3} to \eqref{eqId6}, the generators $\newgen_1,\newgen_2,\newgen_3,\newgen_4$ of $\mathcal{K}$ are in $\newidcomm_\sigma$. Then by \eqref{eqId7}, all the other generators $\newgen_5(k)$ for all $k$ are also elements of $\newidcomm_\sigma$. Then $\mathcal{K}$ is contained in $\newidcomm_\sigma$, and hence $\mathcal{K}=\newidcomm_\sigma$. Therefore, $\F\newmonoid/\newidcomm_\sigma=\F\newmonoid/\mathcal{K}$.
\end{proof}
We use the presentation of $\F\newmonoid/\newidcomm_\sigma$ given in Proposition~\ref{presProp} to construct a \emph{reduction system in $\F\newmonoid$} (i.e., a subset of $\newmonoid\times\F\newmonoid$). Let
\begin{eqnarray}
\alpha & := & \left(AB,\frac{mC-b\algI}{m-1}\right),\\
\beta & := & \left(AC,mCA\right),\\
\gamma & := & \left(BA,\frac{C-b\algI}{m-1}\right),\\
\delta & := & \left(CB,mBC\right),\\
\varepsilon(k) & := & \left(BC^kA , \frac{C^{k+1}-bC^k}{m^k(m-1)}\right),\quad k\in\{1,2,\ldots\}.
\end{eqnarray}
Then $\redsys:=\{\alpha,\beta,\gamma,\delta\}\cup\{\varepsilon(k)\  :\  k\in\{1,2,\ldots\}\}$ is a reduction system in $\F\newmonoid$ for the presentation of $\F\newmonoid/\newidcomm_\sigma$ given in Proposition~\ref{presProp}. We use Bergman's notation for elements of a reduction system, which is as follows. Given an element $\mu$ of a reduction system, we denote the word that serves as the first coordinate of $\mu$ as $W_\mu$ and the second coordinate, a linear combination of words, by $f_\mu$. In the reduction system $\redsys$ for example, we have $W_{\varepsilon(k+1)}=BC^{k+1}A$ and $f_\beta=mCA$.

If $S$ is a reduction system, a word $W\in\newmonoid$ is said to be \emph{$S$-irreducible} whenever $W_\mu$ is not a subword of $W$ for any $\mu\in S$. It is routine to show that a word $W\in\newmonoid$ is $\redsys$-irreducible if and only if 
\begin{eqnarray}
W=C^kA^l,\mbox{ or }\     W=B^lC^k,\label{iffirred}
\end{eqnarray}
for some nonnegative integers $k,l$. 

We now recall the notion of ambiguities. Suppose $S$ is a reduction system in $\F\newmonoid$. An \emph{overlap ambiguity} is a $5$-tuple
\begin{eqnarray}
\overlap=(\mu,\nu,L,X,R)\in S^2\times\newmonoid^3\nonumber
\end{eqnarray}
such that
\begin{eqnarray}
W_\mu = LX,\quad W_\nu = XR.\label{overlapdef}
\end{eqnarray}
If instead of \eqref{overlapdef}, we have $W\mu=X$ and $W_\nu = LXR$, then $\overlap$ is called an \emph{inclusion ambiguity}. The term \emph{ambiguity} means either an overlap ambiguity or an inclusion ambiguity.

It is routine to show that the reduction system $\redsys$ has no inclusion ambiguities, and we describe the inclusion ambiguities in the following. All overlap ambiguities that do not involve an element of $\redsys$ of the form $\varepsilon(k)$ are the following:
\begin{eqnarray}
\overlap_1 & := & \left(\alpha,\gamma,A,B,A\right),\nonumber\\
\overlap_2 & := & \left(\beta,\delta,A,C,B\right),\nonumber\\
\overlap_3 & := & \left(\gamma,\alpha,B,A,B\right),\nonumber\\
\overlap_4 & := & \left(\gamma,\beta,B,A,C\right),\nonumber\\
\overlap_5 & := & \left(\delta,\gamma,C,B,A\right),\nonumber
\end{eqnarray}
while all the overlap ambiguities that depend on an integer parameter are:
\begin{eqnarray}
\overlap_6(k) & := & \left(\alpha,\varepsilon(k),A,B,C^kA\right),\nonumber\\
\overlap_7(k) & := & \left(\delta,\varepsilon(k),C,B,C^kA\right),\nonumber\\
\overlap_8(k) & := & \left(\varepsilon(k),\alpha,BC^k,A,B\right),\nonumber\\
\overlap_9(k) & := & \left(\varepsilon(k),\beta,BC^k,A,C\right).\nonumber
\end{eqnarray}

If $S$ is a reduction system on $\F\newmonoid$, and if $L,R\in\newmonoid$ and $\mu\in S$, then by a \emph{reduction} we mean a linear map $r_{L\mu R}:\F\monoid\rightarrow\F\monoid$ that fixes every element of the basis $\newmonoid$ of $\F\newmonoid$ except $LW_\mu R$, for which the rule $LW_\mu R\mapsto Lf_\mu R$ applies. For simpler notation, we write $r_{\algI\mu R}$, $r_{L\mu\algI}$, $r_{\algI\mu\algI}$ as $r_{\mu R}$, $r_{L\mu}$, $r_{\mu}$, respectively. An overlap ambiguity $\overlap=(\mu,\nu,L,X,R)$ is \emph{resolvable} if there exist compositions $\lambda,\rho$ of reductions such that
\begin{eqnarray}
\lambda(f_\mu R)=\rho(Lf_\nu).\label{overlapresolve}
\end{eqnarray}
We have an analogous definition for the \emph{resolvability of an inclusion ambiguity} in which \eqref{overlapresolve} is replaced by $\lambda(Lf_\mu R)=\rho(f_\nu)$.

For any positive integer $k$, the compositions
\begin{eqnarray}
a_k & := & r_{C^{k-1}\beta}\circ r_{C^{k-2}\beta C}\circ r_{C^{k-2}\beta C^2}\circ\cdots\circ r_{C\beta C^{k-2}}\circ r_{\beta C^{k-1}},\nonumber\\
b_k & := & r_{C^{k-1}\delta}\circ r_{C^{k-2}\delta C}\circ r_{C^{k-2}\delta C^2}\circ\cdots\circ r_{C\delta C^{k-2}}\circ r_{\delta C^{k-1}},\nonumber
\end{eqnarray}
of reductions for the reduction system $\redsys$ shall be of importance in the succeeding lemmas in this section. We note that if $k=1$, we retain the subscript in $b_k=b_1$ so as not to cause confusion with the parameter $b$ of the linearly twisted commutation relation $AB=\sigma(BA)$. It is routine to show that for any positive integers $h,k,l$, the linear maps $a_k$ and $b_k$ satisfy the properties
\begin{eqnarray}
a_k(AC^l) & = & \left\{
    \begin{array}{ll}
      m^kC^kA, & l=k, \\
      AC^l, & l\neq k,
    \end{array}\right.\label{akbkprop1}\\
b_k(C^lB) & = & \left\{
    \begin{array}{ll}
      m^kBC^k, & l=k, \\
      C^lB, & l\neq k.
    \end{array}\right.\label{akbkprop2}\\
a_k(C^hA) & = & C^hA,\label{akbkprop3}\\
b_k(BC^h) & = & BC^h.\label{akbkprop4}
\end{eqnarray}

\begin{lemma}\label{ambigLem} All ambiguities of $\redsys$ are resolvable.
\end{lemma}
\begin{proof} Consider the overlap ambiguity $\overlap_1$. To show that $\overlap_1$ is resolvable, we exhibit compositions $\lambda_1$, $\rho_1$ such that 
\begin{eqnarray}
\lambda_1(f_\alpha A)=\rho_1(Af_\gamma).\label{reductionEq}
\end{eqnarray} We claim that $\lambda_1$ is the empty composition $\id$ and that $\rho_1=r_\beta$. Since $r_\beta$ is a linear map,
\begin{eqnarray}
r_\beta(Af_\gamma)=\frac{r_\beta(AC)-br_\beta(A)}{m-1}.\label{resolve}
\end{eqnarray}
The key step is to examine the words in the right-hand side of \eqref{resolve} and determine if one of them is related to the element of the reduction system $\Gamma$ that appears in the subscript of the reduction involved. We see that $AC=W_\beta$, and so the right-hand side of \eqref{resolve} is a linear combination of $r_\beta(W_\beta)$ and $r_\beta(A)$. But since $AC$ and $A$ are linearly independent in $\F\newmonoid$, we have $A\neq CA = W_\beta$, which means that $r_\beta$ fixes $A$, and by the definition of a reduction, we have $r_\beta(AC)=r_\beta(W_\beta)=f_\beta$. By these observations, \eqref{resolve} becomes
\begin{eqnarray}
r_\beta(Af_\gamma)=\frac{f_\beta-bA}{m-1}=\frac{mCA-bA}{m-1}=f_\alpha A=\id (f_\alpha A),\nonumber
\end{eqnarray}
which implies that we have determined the solution $\lambda_1=\id$, $\rho_1=r_\beta$ to \eqref{reductionEq}, as claimed. The ambiguities $\overlap_2,\ldots,\overlap_5,\overlap_6(k),\ldots,\overlap_9(k)$ can be shown to be resolvable by similar routine computations and arguments. We have an analogue of the equation  \eqref{reductionEq} for each $i\in\{2,3,\ldots,9\}$. We summarize the solutions in the following table.
\begin{center}
\begin{tabular}{|l|c|c|}
\hline
$i$ & $\lambda_i$  & $\rho_i$  \\
 \hline
1 & $\id$ & $r_\beta$ \\
 \hline
2 & $r_{C\alpha}$ & $r_{\alpha C}$ \\
 \hline
3 & $r_{\delta}$  & $\id$ \\
 \hline
4 & $\id$  & $r_{\varepsilon(1)}$ \\
 \hline
5 & $r_{\varepsilon(1)}$ & $\id$ \\
 \hline
6 & $\id$ & $a_k\circ a_{k+1}$ \\
 \hline
7 & $r_{\varepsilon(k+1)}$ & $\id$ \\
 \hline
8 & $b_k\circ b_{k+1}$ & $\id$ \\
 \hline
9 & $\id$ & $r_{\varepsilon(k+1)}$ \\
 \hline
\end{tabular}
\end{center}
More precisely, for each $i\in\{1,2,3,\ldots,9\}$, write the overlap ambiguity as either  $\overlap_i=(\mu,\nu,L,X,R)$ if $i<6$, or $\overlap_i(k)=(\mu,\nu,L,X,R)$ if $i\geq 6$ in which case one of $\mu$, $\nu$ depends on $k$. Then the above table shows the compositions $\lambda_i$, $\rho_i$ of reductions such that $\lambda_i(f_\mu R)=\rho_i(Lf_\nu)$, and this can be verified to be true for each $i\geq 2$ by routine computations. We only note here that in the computations for the case $i\in\{6,7,8,9\}$, the relations \eqref{akbkprop1} to \eqref{akbkprop4} are of significance. By the routine computations just described, all overlap ambiguities of $\redsys$ are resolvable, and since there is no inclusion ambiguity, the proof is now complete.
\end{proof}

\begin{lemma}\label{basisLem} The vectors
\begin{eqnarray}
C^kA^l,\    B^lC^k,\     C^k\quad\in\quad\newAlg, & & (k\in\{0,1,2,\ldots\},\label{theBasis}\\
& & \  \   l\in\{1,2,\ldots\},)\nonumber
\end{eqnarray}
form a basis for $\newAlg$.
\end{lemma}
\begin{proof} Let $\Psi:\F\newmonoid\rightarrow\newAlg$ be the canonical map. Observe that the images of all $\redsys$-irreducible words under $\Psi$ are precisely the vectors \eqref{theBasis}. The only difference between the conditions on the exponents $k,l$ noticeable in comparing \eqref{iffirred} and \eqref{theBasis} is that we avoid the duplication $C^kA^l=\algI=B^lC^k$ when $l=0=k$. Enough has been established in this section for us to invoke the Diamond Lemma \cite[Theorem 1.2]{Ber78}. The only implication we need from \cite[Theorem 1.2]{Ber78} is that: if all the ambiguities of a reduction system $S$ are resolvable and if $\mathcal{K}$ is the ideal of $\F\newmonoid$ generated by all  $W_\mu-f_\mu$ ($\mu\in S$), then the images of all the $S$-irreducible words under the canonical map $\F\newmonoid\rightarrow\F\newmonoid/\mathcal{K}$ form a basis for $\F\newmonoid/\mathcal{K}$. Thus if $S=\redsys$ and $\mathcal{K}=\newidcomm_\sigma$, then by Lemma~\ref{ambigLem}, the vectors \eqref{theBasis} form a basis for $\F\newmonoid/\newidcomm_\sigma$.
\end{proof}

\section{The Lie subalgebra of $\theAlg$ generated by $A,B$}
All computations in this section are in the quotient algebra $\theAlg$. We use the basis \eqref{theBasis} of $\F\newmonoid/\newidcomm_\sigma=\theAlg$ in order to distinguish which elements of $\theAlg$ are Lie polynomials in $A,B$, and which are not. First, we give generalized forms of the relations \eqref{redrel2} and \eqref{redrel4}.
\begin{proposition} For any positive integers $k,l$, the relations
\begin{eqnarray}
A^lC^k & = & m^{kl}C^kA^l,\label{reorderAC}\\
C^kB^l & = &  m^{kl}B^lC^k,\label{reorderBC}
\end{eqnarray}
hold in $\theAlg$.
\end{proposition}
\begin{proof} Use induction on $k,l$.
\end{proof}

\begin{example}\label{tableEx1} The relations \eqref{reorderAC} and \eqref{reorderBC} are reordering formulae for two types of products in $\theAlg$. One is a product of a power of $A$ and a power of $C$ (with positive exponents), and the second type is a product of a power of $C$ and a power of $B$ (also with positive exponents). The importance of these reordering formulae is that given some important expressions like those shown in the left-hand sides of the equations below, we can rewrite such expressions in terms of the basis \eqref{theBasis} of $\theAlg$. That is, by using \eqref{reorderAC}, \eqref{reorderBC}, it is routine to show that the relations
\begin{eqnarray}
\lbrack A,C^x\rbrack & = & (m^{x}-1)C^{x}A,\label{table1}\\
\lbrack A,C^xA^y\rbrack & = & (m^{x}-1)C^{x}A^{y+1},\label{table2}\\
\lbrack C^k,C^xA^y\rbrack & = & (1-m^{ky})C^{k+x}A^{y},\label{table3}\\
\lbrack C^kA^l,C^xA^y\rbrack & = & (m^{lx}-m^{ky})C^{k+x}A^{l+y},\label{table4}\\
\lbrack B,C^x\rbrack & = & (1-m^{x})BC^{x},\label{table5}\\
\lbrack B,B^yC^x\rbrack & = & (1-m^{x})B^{y+1}C^{x},\label{table6}\\
\lbrack C^k,B^yC^x\rbrack & = & (m^{xy}-1)B^{y}C^{k+x},\label{table7}\\
\lbrack B^lC^k,B^yC^x\rbrack & = & (m^{ky}-m^{lx})B^{l+y}C^{k+x}.\label{table8}
\end{eqnarray}
hold in $\theAlg$ for any positive integers $k,l,x,y$. The relations \eqref{table1} to \eqref{table8} shall be important in the proof of our main result.
\end{example}

\begin{example}\label{tableEx2} Reordering in $\theAlg$ will turn out to be more complicated if all of $A,B,C$ appear in an expression that is to be reordered. More of the relations in the presentation of $\theAlg$ in Proposition~\ref{presProp}, will have to be involved, and not just \eqref{redrel2} and \eqref{redrel4} which have been generalized into \eqref{reorderAC} and \eqref{reorderBC}. As an example, by routine computations involving \eqref{redrel1}, \eqref{redrel3}, \eqref{reorderAC}, \eqref{reorderBC}, it can be shown that the relations
\begin{eqnarray}
\lbrack A, B^{y}C^{x}\rbrack & = & \frac{m^y-m^{-x}}{m-1}B^{y-1}C^{x+1} + \frac{m^{-x}-m}{m-1}bB^{y-1}C^{x} ,\label{tablei}\\
\lbrack B, C^{x}A^{y}\rbrack & = & \frac{m^{-x}-m^y}{m-1}C^{x+1}A^{y-1} + \frac{1-m^{-x}}{m-1}bC^{x}A^{y-1},\label{tableii}
\end{eqnarray}
hold in $\theAlg$ for any positive integers $x,y$. The relations \eqref{tablei} and \eqref{tableii} shall also be of significance in the proof of our main result.
\end{example}

The partial descriptions of reordering in $\theAlg$ described in Examples~\ref{tableEx1} and \ref{tableEx2} are sufficient for our purpose of solving the Lie polynomial characterization problem for the algebra $\theAlg$. We now proceed with some technical lemmas.

\begin{lemma}\label{equalexpLem} For any positive integer $n$, there exist $f,g\in\Csub\subseteq\theAlg$ such that
\begin{eqnarray}
(m-1)^nA^nB^n & = & (-1)^nb^n\algI + Cf,\label{divAB}\\
(m-1)^nB^nA^n & = & (-1)^nb^n\algI + Cg.\label{divBA}
\end{eqnarray}
\end{lemma}
\begin{proof} We claim that 
\begin{eqnarray}
(m-1)^nA^nB^n & = & \prod_{i=1}^{n}(m^{i}C-b\algI),\label{equalAB}\\
(m-1)^nB^nA^n & = & \prod_{i=0}^{n-1}(m^{-i}C-b\algI).\label{equalBA}
\end{eqnarray}
We use induction on $n$. If $n=1$, then \eqref{equalAB}, \eqref{equalBA} are immediate consequences of \eqref{redrel1}, \eqref{redrel3}, respectively. Suppose that \eqref{equalAB}, \eqref{equalBA} hold for some positive integer $n$. Multiply both sides of \eqref{equalAB} by $A$ from the left and $B$ from the right to obtain
\begin{eqnarray}
(m-1)^nA^{n+1}B^{n+1} & = & \prod_{i=1}^{n}(m^{i}ACB-bAB).\label{inductAB}
\end{eqnarray}
Use \eqref{redrel2} to replace $AC$ in every occurence of $ACB$ in the right-hand side of \eqref{inductAB}. The result is $h\cdot AB$ where $h$ is a polynomial in $C$. We then use \eqref{redrel1} to replace $AB$ by a polynomial in $C$. After some adjustments in the scalar coefficients, we find that \eqref{equalAB} holds for $n+1$. By similar computations, \eqref{equalBA} can also be shown to hold for $n+1$, and this proves the claim. It is routine to show that the coefficient of $I$ in the right-hand side of \eqref{equalAB} is $(-1)^nb^n$. Thus, every other term in the right hand side of \eqref{equalAB} is a scalar multiple of a power of $C$ with exponent at least 1. By these observations, the right-hand side of \eqref{equalAB} is equal to $(-1)^nb^n\algI + Cf$ for some polynomial $f$ in $C$. This proves \eqref{divAB}. The relation \eqref{divBA} can be proven similarly.
\end{proof}

\begin{lemma}\label{tableLem} If $u,v,x,y$ are positive integers, then $\lbrack C^uA^v,B^yC^x\rbrack$ is an element of the span of
\begin{eqnarray}
C^k,\  C^kA^l,\  B^lC^k,\quad\quad (k,l\in\{1,2,\ldots\}).\label{derLieBasis}
\end{eqnarray}
\end{lemma}
\begin{proof} We first consider the case $v\geq y$. The trick is that $\lbrack C^uA^v,B^yC^x\rbrack$ can be written as
\begin{eqnarray}
\lbrack C^uA^v,B^yC^x\rbrack = C^uA^{v-y}(A^yB^y)C^x-B^yC^{u+x}A^yA^{v-y},\label{vatleasty}
\end{eqnarray}
simply because $v-y$ is a nonnegative integer. Reorder the expressions $C^uA^{v-y}$ and $B^yC^{u+x}$ that appear in the right-hand side of \eqref{vatleasty} using \eqref{reorderAC},\eqref{reorderBC}. We obtain
\begin{eqnarray}
\lbrack C^uA^v,B^yC^x\rbrack = m^{u(y-v)}A^{v-y}P-m^{-y(u+x)}QA^{v-y},\label{vatleasty2}
\end{eqnarray}
where $P=C^u(A^yB^y)C^x$ and $Q=C^{u+x}(B^yA^y)$. By Lemma~\ref{equalexpLem}, $P$ and $Q$ are both polynomials in $C$. Also, the conditions $u>0$ and $x>0$ imply that the powers of $C$ that appear in the linear combinations for $P$ and $Q$ as polynomials in $C$ have exponents that are each at least $1$. This further implies that the first term $m^{u(y-v)}A^{v-y}P$, in the right-hand side of \eqref{vatleasty2}, is a linear combination of products of powers of $A$ and of $C$, which can be reordered using \eqref{reorderAC} such that  $m^{u(y-v)}A^{v-y}P$ can now be written as a linear combination of 
\begin{eqnarray}
C^kA^{v-y},\quad\quad (k\in\{1,2,\ldots\}).\label{CAvectors}
\end{eqnarray}
By inspection, $-m^{-y(u+x)}QA^{v-y}$, the second term in the right-hand side of \eqref{vatleasty2}, is also a linear combination of \eqref{CAvectors}, and by \eqref{vatleasty2}, so is $\lbrack C^uA^v,B^yC^x\rbrack$. The vectors \eqref{CAvectors} are among \eqref{derLieBasis}, and this completes the proof for the case $v\geq y$. By similar computations and arguments, it can be shown that for the case $v<y$, the expression $\lbrack C^uA^v,B^yC^x\rbrack$ is a linear combination of 
\begin{eqnarray}
B^{y-v}C^k,\quad\quad (k\in\{1,2,\ldots\}),
\end{eqnarray}
which are also among \eqref{derLieBasis}.
\end{proof}

\begin{lemma}\label{CpowerLem} For any positive integer $k$, the relation
\begin{eqnarray}
(1-m^{k+1})C^{k+1}=(1-m^k)bC^k-\frac{\left(\left(\ad B\right)\circ\left(\ad C\right)^k\right)(A)}{(1-m)^{k-1}m^{-k}}\label{powerofC}
\end{eqnarray}
holds in $\theAlg$.
\end{lemma}
\begin{proof} Using the relation \eqref{redrel2}, it can be shown by induction that the relation 
\begin{eqnarray}
(\ad C)^k (A) = (1-m)^k C^kA\label{adCA}
\end{eqnarray}
holds in $\theAlg$ for each positive integer $k$. Applying $\ad B$ on both sides of \eqref{adCA}, we get
\begin{eqnarray}
((\ad B)\circ(\ad C)^k) (A) = (1-m)^k (BC^kA-C^kAB).\label{adBadCA}
\end{eqnarray}
Using \eqref{reorderBC}, the right-hand side of \eqref{adBadCA} becomes $$(1-m)^kC^k(m^{-k}BA-AB),$$
in which $BA$ and $AB$ can be replaced by polynomials in $C$ using \eqref{redrel1}, \eqref{redrel3}. In the resulting equation, we solve for $(1-m^{k+1})C^{k+1}$, and the result is \eqref{powerofC}.
\end{proof}

\begin{theorem}\label{mainThm} If $m$ is not a root of unity, then the span of the basis vectors 
\begin{eqnarray}
\algI,\   A^n,\  B^n,\quad\quad (n\geq 2)\label{compLieBasis}
\end{eqnarray}
of $\theAlg$ is a linear complement of the Lie subalgebra of $\theAlg$ generated by $A,B$.
\end{theorem}
\def\candidate{\mathcal{K}}
\def\theLieAlg{\mathcal{L}}
\begin{proof} Denote by $\theLieAlg$ the Lie subalgebra of $\theAlg$ generated by $A,B$, and let $\candidate$ be the span of all vectors in the basis \eqref{theBasis} of $\theAlg$ that are not in \eqref{compLieBasis}. That is, the vectors
\begin{eqnarray}
A,\  B,\  C^k,\  C^kA^l,\  B^lC^k,\quad\quad (k,l\in\{1,2,\ldots\}),\label{LieBasis}
\end{eqnarray}
form a basis for $\candidate$. Then we are done if we are able to show $\candidate=\theLieAlg$. By the minimality (with respect to set inclusion) of $\theLieAlg$ among all Lie subalgebras of $\theAlg$ that contain $A,B$, a sufficient condition for $\candidate$ to be equal to $\theLieAlg$ is:
\begin{enumerate}[label=\upshape(\roman*), leftmargin=*, widest=iii]
    \item $A,B\in\candidate$; and\label{hasAB}
    \item $\candidate$ is a Lie subalgebra of $\theAlg$; and\label{Liesub}
    \item $\candidate\subseteq\theLieAlg$.\label{intheLieAlg}
\end{enumerate}
The condition \ref{hasAB} follows immediately from the definition of $\candidate$. To prove \ref{Liesub}, we show that for any basis vectors $L,R$ of $\candidate$ among \eqref{LieBasis}, $\lbrack L,R\rbrack$ is a linear combination of \eqref{LieBasis}. We summarize all the cases in the following table.
\begin{center}
\begin{tabular}{|c|c|c|c|c|c|}
\hline
 & $A$ & $B$ & $C^x$ & $C^xA^y$ & $B^yC^x$ \\
 \hline
$A$ & $0$ & $C$ & \eqref{table1} & \eqref{table2} & \eqref{tablei} \\
 \hline
$B$ &  & $0$ & \eqref{table5} & \eqref{tableii} & \eqref{table6} \\
 \hline
$C^k$ &  &  & $0$ & \eqref{table3} & \eqref{table7} \\
 \hline
$C^kA^l$ &  &  &  & \eqref{table4} & Lemma~\ref{tableLem}\\
 \hline
$B^lC^k$ &  &  &  & & \eqref{table8}\\
 \hline
\end{tabular}
\end{center}
All the possiblities for $L$ are written in the first column of the above table, while all the possibilities for $R$ are listed in the first row. The cells that are neither in the first column nor in the first row indicate something about $\lbrack L,R\rbrack$. If the cell indicates $0$ or $C$, then $\lbrack L,R\rbrack$ is equal to the content of that cell. If the cell contains an equation number, then that equation shows how $\lbrack L,R\rbrack$ is a linear combination of \eqref{LieBasis}. The most difficult case corresponds to the cell in the table that indicates ``Lemma~\ref{tableLem},'' and this lemma explains why $\lbrack L,R\rbrack$ for this case is a linear combination of \eqref{LieBasis}. The cells below the main diagonal are left empty because of the skew-symmetry of the Lie bracket. This completes our proof of \ref{Liesub}. To prove \ref{intheLieAlg}, we show that for any basis vector $U$ of $\candidate$ among \eqref{LieBasis}, $U\in\theLieAlg$. The proof is clear for the cases $U=A$, $U=B$, $U=C=\lbrack A,B\rbrack$. By Lemma~\ref{CpowerLem}, it can be shown that $C^h\in\theLieAlg$ for any positive integer $h$ using induction. We are left with the last two possible forms of $U$. By this we mean either $U=C^kA^l$ or $U=B^lC^k$ for some positive integers $k,l$. For these cases, observe that by \eqref{reorderAC}, \eqref{reorderBC}, it can be shown that the relations
\begin{eqnarray}
(\ad A)^l(C^k) & = & (m^k-1)^l C^kA^l,\label{adApowerC}\\
(\ad B)^l(C^k) & = & (1-m^k)^l B^lC^k, \label{adBpowerC}
\end{eqnarray}
hold in $\theAlg$ for any positive integers $k,l$, using induction. Since we have already established that $C^k\in\theLieAlg$, the left-hand side of \eqref{adApowerC} is an element of $\theLieAlg$. Since $m$ is not a root of unity, we can isolate $C^kA^l$ in \eqref{adApowerC}, and conclude that $C^kA^l\in\theLieAlg$. By similar arguments and computations, the condition $B^lC^k\in\theLieAlg$ can be deduced from \eqref{adBpowerC}. This proves \ref{intheLieAlg}. Therefore, $\candidate=\theLieAlg$.
\end{proof}

\subsection*{Acknowledgements}
This paper was completed while the author was a post-doctoral research guest in M\"alardalen University, V\"aster\aa s, Sweden, hosted by Prof. Sergei Silvestrov, to  whom the author expresses his thanks, and also to Dr. Lars Hellstr\"om for very  valuable insights. The author acknowledges the unerring support of his home department, the Mathematics and Statistics Department of the College of Science, De La Salle University, Manila. This work was partially supported by a grant from the Commission for Developing Countries of the International Mathematical Union (IMU-CDC).



\begin{thebibliography}{HD82}




\normalsize
\baselineskip=17pt


\bibitem[B78]{Ber78}  G. Bergman,
\emph{The diamond lemma for ring theory},
Adv. Math. 29 (1978), 178--218.

\bibitem[B89]{Bie89} L. Biedenharn,
\emph{The quantum group $SU_q(2)$ and a $q$-analogue of the boson operators},
J.~Phy.~A 22 (1989), L873--L878.

\bibitem[BPS03]{Bla03} P. Blasiak, K. Penson and A. Solomon,
\emph{The general boson normal ordering problem},
Phys.~Lett.~A 309 (2003), 198--205.

\bibitem[BR97]{Bra97} O. Bratteli and D. Robinson,
\emph{Operator algebras and quantum statistical mechanics, vol. 2},
2nd ed., Springer, Berlin Heidelberg, 1997.

\bibitem[C15]{Can15}  R. Cantuba,
\emph{A Lie algebra related to the universal Askey-Wilson algebra},
Matimy\'{a}s Matematika 38 (2015), 51--75.

\bibitem[C17]{Can17} R. Cantuba,
\emph{Lie polynomials in $q$-deformed Heisenberg algebras},
arXiv:1709.02612 (2017).

\bibitem[C18]{Can18} R. Cantuba,
\emph{A Casimir element inexpressible as a Lie polynomial},
arXiv:1810.02554 (2018).

\bibitem[EH67]{Eff67} E. Effros and F. Hahn,
\emph{Locally compact transformation groups and $C^*$-algebras},
Bull.~Amer.~Math.~Soc.~73 (1967), 222--226.

\bibitem[GW54a]{Gaa54a} L. G\aa rding and A. Wightman,
\emph{Representations of the anticommutation relations},
Proc.~Nat.~Acad.~Sci.~U.S.A.~40 (1954), 617--621.

\bibitem[GW54b]{Gaa54b} L. G\aa rding and A. Wightman,
\emph{Representations of the commutation relations},
Proc.~Nat.~Acad.~Sci.~U.S.A.~40 (1954), 622--626.

\bibitem[GK96]{Ger96}  V. Gerdt and V. Kornyak,
\emph{Construction of finitely presented Lie algebras and superalgebras},
J.~Symb.~Comput. 21 (1996), 337--349.

\bibitem[GK97a]{Ger97a}  V. Gerdt and V. Kornyak,
\emph{An algorithm for analysis of the structure of finitely presented Lie algebras},
Discrete Math.~Theor.~Comput.~Sci. 1 (1997), 217--228.

\bibitem[GK97b]{Ger97b}  V. Gerdt and V. Kornyak,
\emph{A program for constructing finitely presented Lie algebras and superalgebras},
Nucl.~Instr.~Meth.~Phys.~Res.~A 389 (1997), 370--373.

\bibitem[GL15]{Goo15} S. Goodenough and C. Lavault,
\emph{Overview of the Heisenberg-Weyl algebra and subsets of Riordan subsgroups},
Electron.~J.~Combin. 22 (2015), P4.16.

\bibitem[H02]{Hel02} L. Hellstr\"{o}m, 
\emph{The diamond lemma for power series algebras}, 
PhD Thesis, Ume\aa\   University (2002).

\bibitem[HS00]{Hel00} L. Hellstr\"{o}m and S. Silvestrov,
\emph{Commuting elements in $q$-deformed Heisenberg algebras}, 
World Scientific, Singapore, 2000.

\bibitem[HS05]{Hel05} L. Hellstr\"{o}m and S. Silvestrov,
\emph{Two-sided ideals in $q$-deformed Heisenberg algebras},
Expo.~Math. 23 (2005), 99--125.

\bibitem[M89]{Mac89} A. Macfarlane,
\emph{On $q$-analogues of the quantum harmonic oscillator and the quantum group $SU(2)_q$},
J.~Phy.~A 22 (1989), 4581--4588.

\bibitem[R93]{Reu93} C. Reutenauer,
\emph{Free Lie algebras}, 
Oxford University Press, New York, 1993.

\bibitem[R81]{Rie81} M. Rieffel,
\emph{$C^*$-algebras associated with irrational rotations},
Pacific J.~Math.~93 2 (1981), 415--429.

\bibitem[Sa]{Shi53}  A. Shirshov,
\emph{Subalgebras of free Lie algebras},  
in: Selected works of A. I. Shirshov (Germany, 2009),   
L.~Bokut et al. (eds.),
Birkh\"{a}user, Switzerland, 2009, 3--13.

\bibitem[Sb]{Shi58}  A. Shirshov,
\emph{On free Lie rings},  
in: Selected works of A. I. Shirshov (Germany, 2009),   
L.~Bokut et al. (eds.),
Birkh\"{a}user, Switzerland, 2009, 77--87.

\bibitem[T11]{Ter11}  P. Terwilliger,
\emph{The universal Askey-Wilson algebra},
SIGMA 7 (2011), 069, 24 pages.

\bibitem[U95]{Ufn95}  V. Ufnarovskij,
\emph{Combinatorial and aymptotic methods in algebra},  
in: Algebra VI (Germany, 1995),
A. Kostrikin and I. Shafarevich (eds.),
Encyclopedia Math. Sci., vol. 57, Springer, New York Berlin Heidelberg, 1994, 1--196.

\bibitem[W13]{Web13} M. Weber,
\emph{On $C^*$-algebras generated by isometries with twisted commutation relations},
J.~Funct.~Anal.~264 (2013), 1975--2004.

\end{thebibliography}
\end{document}